\documentclass[11pt]{article}
\usepackage{amsmath,amssymb,amsfonts,amsthm}
\usepackage{float}
\numberwithin{equation}{section}
\newtheorem{theorem}{Theorem}[section]
\newtheorem{definition}{Definition}[section]
\newtheorem{lemma}[theorem]{Lemma}
\newtheorem{proposition}[theorem]{Proposition}
\newtheorem{corollary}[theorem]{Corollary}

\begin{document}
\begin{center}
{\Large{\textbf{{Characteristics of the Jaco Graph, $J_\infty(a), a \in \Bbb N$}}}} 
\end{center}
\vspace{0.5cm}
\large{
\centerline{(Johan Kok, Paul Fisher, Bettina Wilkens, Mokhwetha Mabula, Vivian Mukungunugwa)\footnote {\textbf {Affiliation of authors:}\\
\noindent Johan Kok (Tshwane Metropolitan Police Department), City of Tshwane, Republic of South Africa\\
e-mail: kokkiek2@tshwane.gov.za\\ \\
\noindent Paul Fisher (Department of Mathematics, University of Botswana), City of Gaborone, Republic of Botswana\\
e-mail: paul.fisher@mopipi.ub.bw\\ \\
\noindent Bettina Wilkens (Department of Mathematics, University of Botswana), City of Gaborone, Republic of Botswana \\
e-mail: wilkensb@mopipi.ub.bw\\ \\
\noindent Mokhwetha Mabula (Department of Mathematics, University of Pretoria), City of Tshwane, Republic of South Africa\\ 
e-mail: mokhwetha.Mabula@up.ac.za\\ \\
\noindent Vivian Mukungunugwa (Department of Mathematics and Applied Mathematics, University of Zimbabwe), City of Harare, Republic of Zimbabwe\\
e-mail: vivianm@maths.uz.ac.zw}}
\vspace{0.5cm}
\begin{abstract}
\noindent  We introduce the concept of a family of finite directed graphs (\emph{order a}) which are directed graphs derived from a infinite directed graph (\emph{order a}), called the $a$-root digraph. The $a$-root digraph has four fundamental properties which are; $V(J_\infty(a)) =\{v_i|i \in \Bbb N\}$ and, if $v_j$ is the head of an edge (arc) then the tail is always a vertexc $v_i, i<j$ and, if $v_k$ for smallest $k \in \Bbb N$ is a tail vertex then all vertices $v_ \ell, k< \ell<j$ are tails of arcs to $v_j$ and finally, the degree of vertex $k$ is $d(v_k) = ak.$ The family of finite directed graphs are those limited to $n \in \Bbb N$ vertices by lobbing off all vertices (and edges arcing to vertices) $v_t, t > n.$ Hence, trivially we have $d(v_i) \leq ai$ for $i \in \Bbb N.$ We present an interesting Lucassian-Zeckendorf result and other general results of interest. It is meant to be an \emph {introductory paper} to encourage exploratory research. \\ \\
\end{abstract}
\noindent {\footnotesize \textbf{Keywords:} Jaco graph, Hope graph, Directed graph, Jaconian vertex, Jaconian set, Number of edges, Shortest path, Fibonacci sequence, Zeckendorf representation}\\ \\
\noindent {\footnotesize \textbf{AMS Classification Numbers:} 05C07, 05C12, 05C20, 11B39} 
\section{Introduction}
\noindent  We introduce the concept of a family of finite Jaco Graphs (\emph{order a}) which are directed graphs derived from the infinite Jaco Graph (\emph{order a}), called the $a$-root digraph. The $a$-root digraph has four fundamental properties which are; $V(J_\infty(a)) =\{v_i|i \in \Bbb N\}$ and, if $v_j$ is the head of an edge (arc) then the tail is always a vertexc $v_i, i<j$ and, if $v_k$ for smallest $k \in \Bbb N$ is a tail vertex then all vertices $v_ \ell, k< \ell<j$ are tails of arcs to $v_j$ and finally, the degree of vertex $k$ is $d(v_k) = ak.$ 
\begin{definition}
The family of infinite Jaco Graphs denoted by $\{J_\infty(a)| a\in \Bbb N\}$ is defined by $V(J_\infty(a)) = \{v_i| i \in \Bbb N\}$, $E(J_\infty(a)) \subseteq \{(v_i, v_j)| i, j \in \Bbb N, i< j\}$ and $(v_i,v_ j) \in E(J_\infty(a))$ if and only if $(a + 1)i - d^-(v_i) \geq j.$
\end{definition}
\begin{definition}
\noindent For $a \in \Bbb N,$ we define the series $(c_{a,n})_{n \in \Bbb N_0}$ by
\begin{center} $c_{a, 0} = 0,$ $c_{a,1} = 1,$ $c_{a,n}= \min \{ k < n| ak + c_{a, k} \geq n \}$ ($n \geq 2)$.
\end{center}
\end{definition}
\noindent The connection between the $a$-root digraph $J_\infty(a)$ and the series $(c_{a,n})$ is explained by the following lemma.
\begin{lemma} 
Consider the Jaco Graph $J_\infty(a)$ and let $n \in \Bbb N$ then the following hold:\\
(a) $d^+(v_n) + d^-(v_n) = an.$\\
(b) $d^-(v_{n+1}) \in \{ d^- (v_n),\, d^-(v_n) + 1 \}.$\\
(c) If $(v_i, v_k) \in E(J_\infty(a))$ and $i < j < k,$ then $(v_j, v_k) \in E(J_\infty(a)).$\\
(d) $d^+ (v_n) = (a-1)n + c_{a,n}.$
\end{lemma}
\begin{proof} 
As $d^+(v_n) = (a+1)n - n - d^-(v_n) ,$ result $(a)$ is obvious. We prove result $(b)$ and $(c)$ simultaneously, using induction on $n.$ First of all, $d^-(v_1) = 0$ implying $d^-(v_2) = 1 = d^-(v_1) + 1.$\\ \\
Let $n \geq 2$ and assume $(b)$ to hold for $m \leq n$ and $(c)$ to hold for $m \leq n-1.$ In particular, $ d^-(v_n) > 0.$ Let $\ell <n$ be minimal with $(v_\ell, v_n) \in E(J_\infty(a)),$ i.e. $(a+1)\ell - d^-(v_\ell) \geq n.$ Let $\ell < j < n.$ By induction,we have
$d^-(v_\ell) \leq d^-(v_j) \leq d^-(v_\ell) + j - \ell$ and $(a+1)j - d^-(v_j)   \geq n.$ Hence, and by choice of $\ell,$ we have $(v_k, v_n) \in E(J_\infty(a))$ if and only if $\ell \leq k < n,$ hence  result $(c)$ is is valid for $n,$ while $d^-(v_n) = n - \ell$ and $d^+(v_n) = an - (n - \ell) = (a-1)n + \ell.$\\ \\
If $(a+1) \ell - d^-(v_\ell) \geq n + 1,$ then $\ell$ is minimal with $(a+1) \ell - d^-(v_\ell) \geq n + 1.$ If $(a+1) \ell - d^-(v_\ell) = n,$ we still, as $\ell + 1 \leq n,$ have $d^- (v_{\ell +1}) \in  \{ d^- (v_\ell),\, d^-(v_{\ell) + 1} \}$ and $(a+1)(\ell + 1) - d^-(v_{\ell +1}) \geq n+1.$  Either way, $(v_{\ell + 1}, v_{n+1}) \in E(J_\infty(a)).$ If $\ell + 1 < j < n,$ then induction yields $d^-(v_{\ell +1}) \leq d^-(v_j) \leq d^- (v_{\ell +1}) + (j - \ell -1) $ and $(a+1) j - d^-(v_j)   \geq n + 1.$ As $d^-(v_n) \leq n-1, (v_n, v_{n+1}) \in E(J_\infty(a)),$ so $(v_k, v_{n+1}) \in E(J_\infty(a))$ whenever $\ell +1 \leq k \leq n.$ Depending on whether $(v_\ell, v_{n+1}) \in E(J_\infty(a))$ or not, we obtain $d^-(v_{n+1}) = n+1 - \ell = (n-\ell) + 1 = d^- (v_\ell) + 1$ or $d^-(v_{n+1}) = n+1 - (\ell+1) = d^- (v_n).$\\ \\
Let $n \geq 3,$ and, as before, choose $\ell$ minimal with $(v_\ell, v_n) \in E(J_\infty(a)).$ We prove result $(d)$ by induction on $n$ and apply arguments very similar to the ones already used: First of all, $d^-(v_1) = 0,$ and $d^+ (v_1) = a = (a-1)1 + c_{a,1}.$ Now let $n > 1,$ and, as before, choose $\ell$  minimal such that $(v_\ell, v_n) \in E(J_\infty(a)).$ By $(a)$ and $(c),$ $d^- (v_n) = n - \ell,$ and $d^+(v_n) = an - n + \ell = (a-1)n + \ell.$ Induction   yields that $d^+(v_k) = a_k$ whenever $k < n.$  The definition of $\ell$ says that  $\ell$ is minimal with $\ell + d^+(v_\ell) = \ell + (a-1) \ell + c_{a,\ell} \geq n,$ which means that $\ell  = c_{a,n}.$
\end{proof}
\begin{corollary}
Note that $(a)$ and $(b)$ of Lemma 1.1 entail that $d^-(v_{n+1}) = (n+1) - c_{a, n+1} \in \{ n - c_{a,n}, n - c_{a, n} + 1\}$ and that $(d)$ then implies that the series  $(c_{a,n})$ are well-defined and ascending, more specifically, $c_{a,n+1} \in \{ c_{a,n}, c_{a,n} + 1 \}$ ($n \in \Bbb N_0$).
\end{corollary}
\begin{lemma}
Let $k \in \Bbb N,$ and $0 \leq b < a .$ Then $c_{a,ak + c_{a,k} - b} = k.$
\end{lemma}
\begin{proof}
Let $ak + c_{a,k} - b = \ell.$ Certainly, $ak + c_{a, k} \geq \ell$ i.e. $c_\ell \leq k.$  On the other hand,
$a(k-1) + c_{a,\, k} = ak + c_{a,\,k} - a < \ell,$ so Corollary 1.2 says $a(k-1) + c_{a, k-1} < \ell$ and $c_\ell = k.$
\end{proof}
\noindent Recall that the \emph{generalised Lucas sequence} $U_n(a,\, -1)$ is defined by \\
\begin{center} $U_0 = 0,$ $U_1 = 1,$\\
$U_{n+1} = a U_n + U_{n-1}.$
\end{center}
It is well known that $U_n = \frac{r^n - s^n}{r-s},$ where $r = \frac{a}{2} + \sqrt{\frac{a^2}{4} +1},$ $s =  \frac{a}{2} - \sqrt{\frac{a^2}{4} + 1}.$

\noindent We are  going to require a probably well-known (and not hard to prove) theorem, which in the case $a = 1$ is known as \emph {Zeckendorf's theorem}.
\begin{lemma}
Let $n \in \Bbb N$ and let $U_0, U_1, \ldots$ be the terms of the Lucas sequence $U(a, -1).$ Then $n$ may be uniquely expressed by a sum
$n = \sum\limits_{i \in \Bbb N} \alpha_i U_i,$ where\\
$0 \leq \alpha_1 < a,$ $0 \leq \alpha_i \leq a$ ($i > 2$), and  $\alpha_i = a$ only if $\alpha_{i-1} = 0$ ($i \in \Bbb N).$
\end{lemma}
\begin{proof}
Induction on $n.$ If $b < a,$ then "$b = b U_1$" is a representation of the required kind, which is clearly unique, as is "$a = U_2.$"\\ \\
Let $0 < b < a,$ and let $n,\, i \in \Bbb N, i \geq 3.$ If $b U_i < n \leq (b+1) U_i,$ then $a U_j < b U_i $ whenever $j < i,$ while $U_{i+1} > a U_i > n.$ If $b U_2 < n \leq (b+1) U_2$ with $b \geq 2,$ then $aU_1 < n< U_3,$ while $(a-1) U_1 < U_2 < U_3.$ It follows that, if $a < n \in \Bbb N,$ then there is a uniquely determined natural number  $i$ for which there exists $\alpha_i \in \{1, \ldots, a \}$ with $(\alpha - 1) U_i < n \leq \alpha_i U_i.$ Now apply induction to  $m = n - \alpha_i U_i.$ \\ \\
Let $n \in \Bbb N,$ and take $n$ as expressed by a sum $n =  \sum\limits_{i \in \Bbb N} \alpha_i U_i$  where the coefficients satisfy the conditions of Lemma 1.3. Define the function $\tau: \Bbb N \rightarrow \{ 0, 1\}$ as follows:\\
 \begin{center} $\tau(n) = \begin{cases}   0&\mbox{if}\, \, \alpha_1 = 0\\
 1 & \mbox{if} \, \alpha_1 > 1\\
 \frac{1 + (-1)^{i+1}}{2} & \mbox{if}\, \, 1 = \alpha_1 = \ldots = \alpha_i,\, \alpha_{i+1} = 0\\
\frac{1 + (-1)^i}{2}  &\mbox{if}\,\, 1 = \alpha_1 = \ldots = \alpha_i < \alpha_{i+1} \end{cases}.$
\end{center}
For $m \in \Bbb N,$ set $c_{a,m} = b_m.$
\end{proof}
\begin{theorem}
Let $n \in \Bbb N,$ $n = \sum\limits_{i \in \Bbb N}\alpha_i U_i$ where the requirements of Lemma 1.4 are assumed to be met. Then $b_n =  \sum\limits_{i \in \Bbb N}\alpha_i U_{i-1} + \tau(n).$
\end{theorem}
\begin{proof}
We proceed by induction on $n.$ Note that $b_1 = \ldots = b_{a-1} = 1 = U_0 + 1,$ while $b_a = b_{a+1} = U_1.$ Now let $a < n =  \sum\limits_{i \in \Bbb N}\alpha_i U_i.$ First suppose that $\alpha_1 = 0$ and $\alpha_2 < a.$ Let $k = \sum\limits_{i \in \Bbb N} \alpha_i U_{i-1}.$ Letting  $\beta_i = \alpha_{i+1} ( i \in \Bbb N),"k =   \sum\limits_{j \in \Bbb N}\beta_j U_j$ " is a representation meeting the conditions of Lemma 1.4, hence induction yields $b_k = \sum\limits_{i \in \Bbb N} \alpha_i U_{i-2} + \tau(k).$ Now $n = a( \sum\limits_{i \in \Bbb N} \alpha_i U_{i-1}) + \sum\limits_{i \in \Bbb N} \alpha_i U_{i-2} = k + b_k - \tau(k),$ and, as $\tau(k) < a,$ Lemma 1.3 yields $k = b_n.$\\ \\
Now assume that $\alpha_1 = 0$ and $\alpha_2 = a.$ We consider the case $a = 1$ separately and first.If $a = 1,$ then $\alpha_2 = 1$ forces $\alpha_3 = 0,$ and $ \sum\limits_{i > 3} \alpha_i U_i$ is the Zeckendorf representation of $n-1.$ Let $k = b_{n-1}.$ Via induction,$k =  \sum\limits_{i > 3} \alpha_i U_{i-1},$ and, as the sum on the left is the Zeckendorf representation of $k,$ induction yields $b_k =  \sum\limits_{i > 3} \alpha_i U_{i-2}.$ It follows that $k + b_k = \sum\limits_{i > 3} \alpha_i U_i = n-1,$ whence $b_n > b_{n-1},$ which, as we have seen, means that $b_n = b_{n-1} + 1 = \sum\limits_{i > 3} \alpha_i U_{i-1} + 1 = \sum\limits_{i \geq 3} \alpha_i U_{i-1} + U_1.$ At this stage, the theorem has been proved in the case $a = 1,$ so we will assume $a>1$ from now on.\\ \\
Suppose that  $\alpha_2 = a$  and let $j$ be maximal with $\alpha_{2j} = a.$ Then $n = r + s,$ where $r = \sum\limits_{i \geq 2j + 1} \alpha_i U_i$ and $s = \sum\limits_{i =1}^j a U_{2i} = U_{2j +1} - 1.$ It follows from our choice of $j$ that both $\alpha_{2j+ 1}$ and $\alpha_{2(j+1)}$ are less than $a.$  Accordingly,  "$n+1 = \sum\limits_{i \geq 2(j+1)} \alpha_i U_i + (\alpha_{2j+1} + 1) U_{2j+1}$" is a sum representation satisfying the conditions of Lemma 1.4. Let $k = \sum\limits_{i \geq 2(j+1)} \alpha_i U_{i-1} + (\alpha_{2j+1} + 1)U_{2j}.$ Then $k < n,$ and induction yields that $b_k = \sum\limits_{i \geq 2(j+1)} \alpha_i U_{i-2} + (\alpha_{2j+1} + 1) U_{2j-1}.$ Now $ak + b_k = n+1,$ and Lemma 1.3 says $b_n = b_{n+1} = k.$ On the other hand, $U_{2j} = \sum\limits_{i = 2}^j a U_{2i -1},$ so $k = \sum\limits_{i \in \Bbb N} \alpha_i U_{i-1}.$\\ \\
Now assume that $\alpha_1 > 0$ and let $m = \sum\limits_{i \geq 2} \alpha_i U_{i-1}.$ As $\alpha_1 > 0,$ we have $\alpha_2 < a,$ and the displayed sum representation of $m$ satisfies the conditions of Lemma 1.4.\\ \\
Via induction, $b_m \in \{ \sum\limits_{i \geq 2} \alpha_i U_{i-2},\,\sum\limits_{i \geq 2} \alpha_i U_{i-2}+ 1\},$ hence
$ n - \alpha_1 \leq am + b_m \leq n - \alpha_1 + 1$ \hfill $(\ast)$ \\ \\
As $\alpha_1 < a,$ the inequality $(\ast)$ implies that $a(m+1) + b_{m+1} \geq a(m+1) + b_m = a_m + b_m + a > n.$ Thus $b_n \in \{ m,\, m+1 \}.$ If $\alpha_1 > 1,$ then $\tau(n) = 1$ and  $n - \alpha_1 + 1 < n,$ whence $b_n = m+1.$\\ \\
We are left to deal with the case  $\alpha_1 = 1$ which we shall subdivide further, as follows:\\ \\
If $\alpha_1 = 1$ and $\alpha_2 = 0,$ then $\tau(n) = 1,$ while induction yields that $b_m =  \sum\limits_{i \geq 2} \alpha_i U_{i-2},$ such that $am + b_m = \sum\limits_{i \geq 2} \alpha_i U_i = n-1$ and Corollary 1.2 implies that $b_n = b_{n-1} + 1 = m + \tau(n).$\\ \\
If $\alpha_1 = 1 < \alpha_2,$ then $\tau(n) = 0,$ while induction yields $b_m =  \sum\limits_{i \geq 2} \alpha_i U_{i-2} + 1,$
hence $am + b_m = n$ and $m = b_n$ by Lemma 1.3.\\ \\
If, finally, $\alpha_1 = 1 = \alpha_2,$ then $\tau(m) = 0$ if and only if $\tau(n) = 1,$ and induction yields $am + b_m = n-1$ if and only if $\tau(n) = 1$ and  $am + b_m = n$ if and only if $\tau(n) = 0.$
\end{proof}
 
\section{Finite Jaco Graphs $\{J_n(a)| a, n\in \Bbb N\}$} 
\noindent The family of finite Jaco Graphs are those limited to $n \in \Bbb N$ vertices by lobbing off all vertices (and edges arcing to vertices) $v_t, t > n.$ Hence, trivially we have $d(v_i) \leq ai$ for $i \in \Bbb N.$
\begin{definition}
The family of finite Jaco Graphs denoted by $\{J_n(a)| a, n\in \Bbb N\}$ is the defined by $V(J_n(a)) = \{v_i| i \in \Bbb N, i \leq n \}$, $E(J_n(a)) \subseteq \{(v_i, v_j)| i,j \in \Bbb N, i< j \leq n\}$ and $(v_i,v_ j) \in E(J_n(a))$ if and only if $(a + 1)i - d^-(v_i) \geq j.$
\end{definition} 
\begin{definition}
The set of vertices attaining degree $\Delta (J_n(a))$ is called the Jaconian vertices of the Jaco Graph $J_n(a),$ and denoted, $\Bbb{J}(J_n(a))$ or, $\Bbb{J}_n(a)$ for brevity.
\end{definition}
\begin{definition}
The lowest numbered (indiced) Jaconian vertex is called the prime Jaconian vertex of a Jaco Graph.
\end{definition}
\begin{definition}
If $v_i$ is the prime Jaconian vertex, the complete subgraph on vertices $v_{i+1}, v_{i+2}, \\
\cdots,v_n$ is called the Hope subgraph of a Jaco Graph and denoted,  $\Bbb{H}(J_n(a))$ or, $\Bbb{H}_n(a)$ for brevity.
\end{definition}

\noindent {\bf Property 1:} From the definition of a Jaco Graph $J_n(a),$ it follows that, if for the prime Jaconian vertex $v_i,$ we have $d(v_i)=ai$ then in the underlying Jaco graph we have $ d(v_m)=am$ for all $m\in\{1,2,3,\cdots,i\}.$\\ \\
{\bf Property 2:} From the definition of a Jaco Graph $J_n(a),$ it follows that $\Delta (J_k(a))\leq \Delta (J_n(a))$ for all $k\leq n.$\\ \\
{\bf Property 3:} From the definition of a Jaco Graph $J_n(a),$ it follows that the lowest degree attained by all Jaco Graphs is $0 \leq \delta(J_n(a)) \leq a.$\\ \\
{\bf Property 4:} The $d^-(v_k)$ for any vertex $v_k$ of a Jaco Graph $J_n(a),~ n\geq k$ is equal to $d(v_k)$ in the underlying Jaco Graph $J_k(a).$\\ \\
\begin{lemma}
For the Jaco Graphs $J_i(a), i\in \{1,2,3,...,a+1\}$ we have $\Delta (J_ i(a)) = i-1$ and $\Bbb {J}(J_i(a)) = \{v_k| 1\leq k \leq i\} =V(J_i(a)).$
\end{lemma}
\begin{proof}
From definition 2.1 it follows that if $m = a+1$ then, $((a+1) +1).1 - d^-(v_1) > (a+1)$ so the edges $(v_1,v_i), i= 2,3, ..., (a+1)$ exist. It then follows that all edges $(v_i,v_j), i<j$ exist. So the underlying graph of $J_{a+1}(a)$ is the complete graph $K_{a+1}.$ Since $\Delta (J_{a+1}(a)) = (a+1) - 1 = a$ and we have $d(v_i) = a$ for all vertices in $K_{a+1},$ it follows that $\Delta (J_{a+1}(a)) = a$ and $\Bbb {J}(J_{a+1}(a)) = \{v_k| 1\leq k \leq a\} =V(J_{a+1}(a)).$  The result follows similarly for $m < a+1.$
\end{proof} 
\begin{lemma}
If in a Jaco Graph $J_n(a),$ and for smallest $i,$ the edge $(v_i, v_n)$ is defined, then $v_i$ is the prime Jaconian vertex of $J_n(a).$
\end{lemma}
\begin{proof}
If in the construction of a Jaco Graph $J_n(a),$ and for smallest $i,$ the edge $(v_i, v_n)$ is defined, we have that in the underlying graph of $J_n(a)$, $d(v_i) \leq ai$ and $d(v_j) \leq d(v_i)$ for all $j > i$. So it follows that $d(v_i) = \Delta(J_n(a))$ so $v_i$ is the prime Jaconian vertex of $J_n(a).$ 
\end{proof}
\begin{lemma}
For all Jaco Graphs $J_n(a),~n\geq2$ and, $v_i, v_{i-1}\in V(J_n(a))$ we have that in the underlying graph $|(d(v_i) - d(v_{i-1})|\leq a.$ 
\end{lemma}
\begin{proof}
Consider the Jaco Graph $J_n(a),~n\geq2.$ The result is trivially true for all vertices $v_1,v_2,v_3,\cdots,v_k$ if $v_k$ is the prime Jaconian vertex of $J_n(a).$ Now consider the Hope Graph $\Bbb{H}(J_n(a)).$ All vertices of $\Bbb{H}(J_n(a))$ have equal degree so the result holds for the Hope Graph per se. Furthermore if a vertex $v_j,~(k+1)\leq j\leq n$ is linked to a vertex $v_t,~1\leq t\leq k$ then all vertices $v_l,~(k+1)\leq l< j$ are linked to $v_t$ which implies 
$|d(v_j)-d(v_l)|=0\leq1 \leq a$ and $|(d(v_{j+1})-d(v_j)|\leq1 \leq a.$
\end{proof}
\noindent Note that $\Delta (J_n(a))$ might repeat itself as $n$ increases to $n+1$ but on an increase we always obtain $\Delta (J_n(a))+1$ before $\Delta (J_n(a))+2.$
\begin{theorem}
The Jaco Graph $J_k(a)$, $k=a(a+1) +1$ is the smallest Jaco Graph in $\{J_n(a)|a, n \in \Bbb {N}\}$ which has $\Delta (J_k(a)) = a(a+1)$ and $\Bbb {J}(J_k(a)) = \{v_{a+1}\}.$ 
\end{theorem}
\begin{proof}
For $a = 1$, the graph $J_3(1)$ is clearly the smallest Jaco Graph for which $\Delta (J_3(1)) =1(1 +1) = 2$ and $\Bbb J(J_3(1)) = \{v_{(1+1)}\} = \{v_2\}.$ So the result holds for $a = 1.$\\ \\
Assume the result holds for $a= m.$ So for the Jaco Graph $J_l(m), l = m(m + 1) + 1$ we have that $\Delta (J_l(m)) = m(m+1)$ and $\Bbb J(J_l(m)) = \{v_{(m+1)}\}.$ \\ \\
Now consider the Jaco Graph $J_k(m+1).$ In the Jaco Graph $J_l(m)$ the vertex $v_{(m+1)+1} = v_{(m+2)}$ has $d(v_{(m+2)}) = d(v_{(m+1)}) - 1.$ So in constructing the Jaco Graph $J_l(m+1),$ amongst others the edge $(v_1, v_{(m+2)})$ is linked.  So at least $v_{(m+1)}, v_{(m+2)} \in \Bbb J(J_l(m+1)).$ So $d(v_{(m+2)}) = m(m+1).$ If follows that the minimum number of additional vertices (smallest Jaco Graph) say $t,$ to be added to $J_l(m+1)$ to obtain $d(v_{(m+2)}) = (m+1)(m+2)$ and $\Bbb J(J_{(l+t)}(m+1)) = \{v_{(m+2)}\}$ in $J_{(l+t)}(m+1)$ is given by $t = (m+1)(m+2) - m(m+1) = 2(m+1).$ 
The number of vertices of $J_l(m)$ is given by $m(m+1) +1.$ Now $l + t = (m(m+1) + 1) +2(m+1) = (m+1)(m+2) + 1.$ \\ \\
Clearly at least $v_{(m+2)} \in \Bbb J(J_k(m+1)), k= l+t,$ and $v_{(m+1)} \notin \Bbb J(J_k(m+1))$ since $m(m+1) < (m+1)(m+2).$ In the construction of the Jaco Graph $J_l(m+1),$ the edge $(v_1, v_{(m+3)})$ was not linked so $d(v_{(m+3)}) < d(v_{(m+2)})$ in $J_l(m+1).$ So it follows that $d(v_{(m+3)}) < d(v_{(m+2)})$ in $J_k(m+1), k = l + t.$ The latter implies that $\Bbb J(J_k(m+1)) = \{v_{(m+2)}\}.$\\ \\
Hence the result holds for $a = m + 1$ implying it holds in general.
\end{proof}
 
\section{Number of Edges of the Finite Jaco Graphs $\{J_n(a)| n\in \Bbb N\}$}
\noindent Note that Theorem 3.7 combined with Binet's formula amounts to a closed formula for $d^+(v_n).$
It is hoped that as a special case, a closed formula can be found for the number of edges of a finite Jaco Graph $J_n(a).$ However, the algorithms discussed in Ahlbach et al.[4] suggest this might not be possible.
\begin{proposition}
The number of edges of a Jaco Graph $J_{m}(a) = \frac{1}{2}m(m-1)$ if $m \leq a+1.$
\end{proposition}
\begin{proof}
From definition 2.1 it follows that if $m = a+1$ then, $((a+1) +1).1 - d^-(v_1) > (a+1)$ so the edges $(v_1,v_i), i= 2,3, ..., (a+1)$ exist. It then follows that all edges $(v_i,v_j), i<j$ exist. So the underlying graph of $J_{a+1}(a)$ is the complete graph $K_{a+1}$ hence, $\epsilon (J_{a+1}(a)) =  \frac{1}{2}a(a+1) = \frac{1}{2}m(m-1).$  The result follows similarly for $m < a+1.$
\end{proof}
\begin{theorem}
If for the Jaco Graph $J_n(a),$ we have $\Delta (J_n(a)) = k,$ then  $\epsilon(J_n(a)) = \epsilon(\Bbb{H}(J_n(a))) + \sum \limits_{i=1}^{k}d^+(v_i). $
\end{theorem}
\begin{proof} 
For $n=1,$ $d^+(v_1) = 0$ and $J_1(a)$ is the edgeless graph on vertex $v_1$ whilst $\Bbb{H}(J_1(a)),$ is an empty graph so $E(\Bbb{H}(J_1(a))) =\emptyset,$ implying  $\epsilon(\Bbb{H}(J_1(a))) = 0.$ Thus the result holds.\\ \\
For $n = 2,$ the Jaco Graph $J_2(a)$ has the prime Jaconian vertex $v_1$ and $d^+(v_1) = 1.$ $\Bbb{H}(J_2(a))$ is the null graph on vertex $v_2$ so $E(\Bbb{H}(J_2(a)))=\emptyset,$ implying  $\epsilon(\Bbb{H}(J_2(a))) = 0.$ Thus the result holds.\\ \\
Now assume it holds for all vertices $v_i,~i\leq k-1.$ Thus vertex $v_k$ has attained its in-degree $d^-(v_k).$ To attain $d(v_k) = \Delta (J_n(a)),$ exactly $\Delta (J_n(a))-d^-(v_k) = d^+(v_k)$ edges can be linked additionally. So the result holds for vertices $v_1, v_2, v_3,\cdots, v_k.$\\ \\
Clearly we also have $E(\Bbb{H}(J_n(a)))\subset E(J_n(a)).$ Hence, $\epsilon(J_n(a))= \epsilon(\Bbb{H}(J_n(a))) + d^+(v_1) + d^+(v_2) + d^+(v_3) +\cdots+ d^+(v_k).$ \\
Since it is known that $\epsilon(\Bbb{H}(J_n(a))) = 1 + 2 + 3 +\cdots + (n - \Delta (J_n(a))-1)$ the result can be written as  $\epsilon(J_n(a)) =\frac{1}{2}(n-k)(n-k-1)+ \sum \limits_{i=1}^{k}d^+(v_i). $
\end{proof}
\begin{corollary}
The number of edges of a Jaco Graph $J_n(a)$ having vertex $v_i$ as the prime Jaconian vertex, can also be expressed recursively as
\begin{equation*} 
\epsilon (J_{n+1}(a)) =
\begin{cases}
\epsilon(J_n(a)) - i + n &\text {if $d(v_i) = ai$}\\
\epsilon(J_n(a)) - i + (n+1) & \text {if $d(v_i) < ai$}\\ 
\end{cases}
\end{equation*} 
\end{corollary}
\begin{proof}
Consider the Jaco Graph $J_n(a)$ having vertex $v_i$ as the Jaconian vertex.\\ \\
Case1:  If $ d(v_i) = ai$ and the vertex $v_{(n+1)}$ is added to construct $J_{(n+1)}(a)$ only the edges 
$(v_{(i+1)}, v_{(n+1)}), .... (v_n, v_{(n+1)})$ can be linked additionally. This amounts to $(n - i)$ edges.\\ \\
Case 2: If $ d(v_i) < ai$ and the vertex $v_{(n+1)}$ is added to construct $J_{(n+1)}(a)$ only the edges 
$(v_i, v_{(n+1)}), (v_{(i+1)}, v_{(n+1)}), .... (v_n, v_{(n+1)})$ can be linked additionally. This amounts to $(n - i +1)$ edges. 
\end{proof}
\begin{lemma}
\noindent We find for $a=1$ and the series $(a_n)_{n \in \Bbb N}$ defined by
\begin{center} $a_0 = 0, a_1 = 1, a_n = \min \{ k < n| k + a_k \geq n \}$ ($n \geq 2)$.
\end{center}
that Lemma 1.1 changes to: \\
\noindent (a) $d^+(v_n) + d^-(v_n) = n.$\\
(b) $d^-(v_{n+1}) \in \{ d^- (v_n),\, d^-(v_n) + 1 \}.$\\
(c) If $(v_i, v_k) \in E(J_\infty(1))$ and $i < j < k,$ then $(v_j, v_k) \in E(J_\infty(1)).$\\
(d) $d^+ (v_n) = a_n.$
\end{lemma}
\begin{corollary}
Note that $(a))$ and $(c)$ above entail that $d^+(v_{n+1}) = n+1- d^-(v_{n+1}) \in \{n-d^-(v_n), n-d^-(v_n) +1\}$ and that $(d)$ then implies tat the series $(a_n)$ is well defined and ascending, more specifically, $a_{n+1} \in \{a_n, a_n+1\},$ $(n\in \Bbb N_0).$
\end{corollary}
\begin{lemma}
Let $i \in\Bbb N.$ Then $d^+(v_{i+d^+(v_i)}) = i = d^+(v_{i+d^+(v_{i+d^+(v_{i-1})})}).$
\end{lemma}
\begin{proof}
let $i+d^+(v_i) = k.$ Certainly, $i+d^+(v_i) \geq k,$ so $d^+(v_k) \leq i.$ From Lemma 1.1, $(a=1)$ it follows that $i-1+d^+(v_{i-1}) \leq i-1+ d^+(v_i) < i+d^+(v_i)$ so $d^+(v_k) \geq i.$ Let $\ell = i+ d^+(v_{i-1}).$ Since $d^+(v_i) \geq d^+(v_{i-1})$ we have, $d^+(v_\ell) \leq i$ and since, $i-1+ d^+(v_{i-1}) < \ell$ we have, $d^+(v_\ell) =i.$
\end{proof}
\begin{theorem}
(Bettina's Theorem)\footnote{Also see: [5] Kok et al. \emph{Characteristics of Finite Jaco Graphs, $J_n(1), n \in \Bbb N.$}}: Let $\Bbb{F} = \{f_0, f_1,f_2, ...\}$ be the set of Fibonacci numbers and let $n=f_{i_1} + f_{i_2} + ... + f_{i_r}, n\in \Bbb N$ be the Zeckendorf representation of $n.$ Then 
\begin{center}
$d^+(v_n) = f_{i_1-1} + f_{i_2-1} + ... +f_{i_r-1}.$
\end{center}
\end{theorem}
\begin{proof}
Through induction we have that first of all, $1=f_2$ and $d^+(v_1) =1=f_1.$ Let $2 \leq n = f_{i_1} + f_{i_2} + ... + f_{i_r}$ and let $k = f_{i_1-1} + f_{i_2-1} + ... + f_{i_r-1}.$ If $i_r \geq 3,$ then $k = f_{i_1-1} + f_{i_2-1} + ... + f_{i_r-1}$ is the Zeckendorf representation of $k$, such that induction yields $d^+(v_k) = k = f_{i_1-2} + f_{i_2-2} + ... + f_{i_r-2}.$ Since $k + d^+(v_k)  = f_{i_1-1} + f_{i_1-2} + f_{i_2-1} + f{i_2-2} + ... f_{i_r-1} + f_{i_r-2} = f_{i_1} + f_{i_2} + ... f_{i_r} = n,$ read with Lemma 3.6 yields $d^+(v_n) = k.$ \\ \\
Finally consider $n=f_{i_1} + f_{i_2} + ... + f_{i_r}, i_r =2.$ Note that $n >1$ implies that $i_{r-1} \geq 4$ and that the Zeckendorf representation of $n-1$ given by $n-1 = f_{i_1} + f_{i_2} + ... + f_{i_{r-1}}.$ Let $k= d^+(v_{n-1}).$ Through induction we have that, $k = f_{i_1-1} + f_{i_2-1} + ... + f_{i_{r-1}-1},$ and since $i_{r-1} \leq 4,$ this is the Zeckendorf representation of $k.$ Accordingly, $d(v_k) = f_{i_1-2} + f_{i_2-2}+ ... + f_{i_{r-1}-2},$ and $k+ d^+(v_k) =  f_{i_1-1} + f_{i_1-2} + f_{i_2-1} + f{i_2-2} + ... f_{i_{r-1}-1} + f_{i_{r-1}-2} = n-1.$ It follows that $d^+(v_n) > k = d^+(v_{n-1}).$ From Corollary 3.4 it follows that $d^+(v_n) = k+1 = (f_{i_1-1} + f_{i_2-1} + ... + f_{i_{r-1}-1}) + f_1 = f_{i_1-1} + f_{i_2-1} + ... + f_{i_r-1}.$
\end{proof}

\section{Number of Shortest Paths in the Finite Jaco Graphs $\{J_n(a)| n\in \Bbb N\}$}
\begin{definition}
Liz numbers are the family of numbers defined by $\Bbb L = \{\Bbb L_a |B_0 = 0, B_1 = 1, B_2 = 1, B_i = aB_{i-1} + B_{i-2}, a,i \in \Bbb N, i \geq 3\}.$ 
\end{definition}
\begin{definition}
The set of distance-root vertices of the Jaco Graph $J_n(a),$ is the set \\
$\{v_{B_2}, v_{B_3},v_{B_4},\cdots,v_{B_{j} < n}$,$\cdots,v_n\}$ and denoted, $\Bbb{D}(J_n(a))$ or, $\Bbb{D}_n(a)$ for brevity.
\end{definition}
\noindent {\bf Property 5:} From definitions 4.1 and 4.2 it follow that besides possibly $v_n,$ all other distance-root vertices of the Jaco Graph $J_n(a),~n\in\Bbb{N},$ have Liz number indices. \\ \\
\noindent {\bf Property 6:} The set of Fibonacci numbers $\Bbb F \in \Bbb L,$ since $a=1$ for $\Bbb F.$ \\ \\
\noindent Note: Reader should note the subtle difference between the $(a,1)$-Fibonacci sequence defined in Kalman et al.[3], and the definition of Liz numbers finding application in this paper.
\begin{lemma}
In a Jaco Graph $J_{n}(a)$ we have for smallest $k,$ such that   $k+d^+(v_k)\geq i,$ that $d_{J_n(a)}(v_1, v_i) = d_{J_n(a)}(v_1, v_k) + 1.$
\end{lemma}
\begin{proof}
If in a Jaco Graph $J_{m}(a)$ we have that the edge $(v_1, v_i)$ exists, the result holds because $d_{J_n(a)}(v_1, v_1) = 0.$ Otherwise find smallest $k$ for which $k+d^+(v_k)\geq i.$ It follows that for all $j < k$ we have $k+d^+(v_j) < i.$ So a path via such $v_j$ will have length at least, $d_{J_n(a)}(v_1, v_j) + 2.$ However, the path via $v_k$ has length $d_{J_n(a)}(v_1, v_j) + 1.$ Hence, $d_{J_n(a)}(v_1, v_i) = d_{J_n(a)}(v_1, v_k) + 1.$
\end{proof}
\begin{definition}
Let the number of distinct shortest paths between $v_1$ and $v_n$ in the Jaco Graph $J_n(1)$ be denoted by $\psi(v_n).$
\end{definition}
\begin{proposition}
Consider vertex $v_j, ~j\geq 1$ in $J_n(1)$. The shortest path between $v_1$ and $v_j$ is unique if and only if, $d^+(v_j)\in\Bbb{F}.$ 
\end{proposition}
\begin{proof}
Since any vertex is inherintly linked to itself the result holds for $j=1.$ We have that the set of Fibonacci numbers are the Liz numbers for $a = 1,$ so the result follows for all $v_j, j > 2$ and  $j\in \Bbb {F}$ since definition 4.1 ensures that $d^+(v_j)\in\Bbb{F}.$ and  definition 4.2 ensures that the shortest path is unique.  \\ \\
Assume that there exists a $j \notin \Bbb {F}$ such that the shortest path between $v_1$ and $v_j$ is unique. It implies that the smallest $l$ for which the edge $(v_l, v_j)$ exists, is the largest $f_i$ for which $(v_{f_i}, v_j)$ exists hence, $l = f_i$. Following from Lemma 1.1 ($a=1$), $d^+(v_{l+d^+(v_l)}) = l = d^+(v_{l+d^+(v_{l-1})})$ it follows that $d^+(v_j) = l \in\Bbb{F}$. \\ \\
Assume that there exists a $j \notin \Bbb {F}$ such that $d^+(v_j)\in\Bbb{F}$. It implies that $j$ is the largest natural number for which $(v_{f_i}, v_j)$ exists and $f_i$ being the largest Fibonacci number less than $j$. Clearly, a minimal path between $v_1$ and $v_j$ via $v_k$, $k<v_{f_i}$ or $k>v_{f_i}$ will have lenght greater than the minimal path via $v_{f_i}$. Hence, the shortest path between $v_1$ and $v_j$ is unique.
\end{proof}
\begin{proposition}
Consider vertex $v_j, ~j\geq 1$ in $J_n(1)$.  If and only if  $d^+(v_j)\notin\Bbb{F},$ then $\psi(v_j) = \sum \limits_{i=l}^{f_t}\psi(v_i)$, with $f_t$ the largest Fibonacci number less than $j$ and $l$ the smallest integer such that the edge $(v_l, v_j)$ exists.
\end{proposition}
\begin{proof}
Let $d^+(v_j)\notin\Bbb{F},$ then clearly $d(v_1, v_l) \leq d(v_1, v_j)$ for $l<j$. So if the edge $(v_l, v_j)$ exists Lemma 4.1 states that, $d(v_1, v_j) = d(v_1, v_l) + 1.$ So the number of shortest paths between $v_1$ and $v_j$ via vertex $v_l$ is given by $\psi(v_l)$. Through recursion the result  $\psi(v_j) = \sum \limits_{i=l}^{f_t}\psi(v_i)$  with $f_t$ the largest Fibonacci number less than $j$ and $l$ the smallest integer such that the edge $(v_l, v_j)$ exists, follows. \\ \\
The converse follows from Proposition 4.2.
\end{proof}
\begin{proposition}
Let $k \geq 7$ and $f_i < k \leq f_{i+1}$, and $f_i,f_{i+1} \in \Bbb{F}.$ If and only if $ d^+(v_k)$ is non-repetitive (meaning $d^+(v_{k-1}) \neq d^+(v_k) \neq d^+(v_{k+1})$) then $\psi(v_k)$ is non-repetitive (meaning $\psi(v_{k-1}) \neq \psi(v_k) \neq \psi(v_{k+1})$).
\end{proposition}
\begin{proof}
(Conjectured)
\end{proof}
\noindent [Open problem: Does a closed formula exist for the number of edges, $\epsilon(J_n(1)).$] \\ 
\noindent [Open problem: Prove Proposition 4.4] \\ 
\noindent [Open problem: Find $Card$ $\Bbb J(J_n(a))$ in general.] \\ \\
\textbf{\emph{Open access:}} This paper is distributed under the terms of the Creative Commons Attribution License which permits any use, distribution and reproduction in any medium, provided the original author(s) and the source are credited. \\ \\ \\
References \\ \\
$[1]$  Bondy, J.A., Murty, U.S.R., \emph {Graph Theory with Applications,} Macmillan Press, London, (1976). \\
$[2]$ Zeckendorf, E., \emph {Repr\'esentation des nombres naturels par une somme de nombres de Fibonacci ou de nombres de Lucas,} Bulletin de la Soci\'et\'e Royale des Sciences de Li\`ege, Vol 41, (1972): pp 179-182. \\ 
$[3]$ Kalman, D., Mena, R., \emph{The Fibonacci Number - Exposed,} Mathematics Magazine, Vol 76, No. 3, (2003): pp 167-181. \\ 
$[4]$ Ahlbach, C., Usatine, J., Pippenger, N., \emph{Efficient Algorithms for Zerckendorf Arithmetic,} Fibonacci Quarterly, Vol 51, No. 13, (2013): pp 249-255. \\
$[5]$ Kok, J., Fisher, P., Wilkens, B., Mabula, M., Mukungunugwa, V., \emph{Characteristics of Finite Jaco Graphs, $J_n(1), n \in \Bbb N$}, arXiv: 1404.0484v1 [math.CO], 2 April 2014. \\ \\ 
\noindent Acknowledgement will be given to colleagues for preliminary peer review and other contributions on the content of this paper during the preprint arXiv publication term:\footnote{} \\ \\
\noindent {\scriptsize [Remark: The concept of Jaco Graphs followed from a dream during the night of 10/11 January 2013 which was the first dream Kokkie could recall about his daddy after his daddy passed away in the peaceful morning hours of 24 May 2012, shortly before the awakening of Bob Dylan, celebrating Dylan's 71st birthday]}\\ \\
\end{document}